\documentclass[12pt, reqno]{amsart}
\usepackage{amsmath, amsthm, amscd, amsfonts, amssymb, graphicx, color}
\usepackage[bookmarksnumbered, colorlinks, plainpages]{hyperref}
\hypersetup{colorlinks=true,linkcolor=red, anchorcolor=green, citecolor=cyan, urlcolor=red, filecolor=magenta, pdftoolbar=true}

\newtheorem{theorem}{Theorem}[section]
\newtheorem{lemma}[theorem]{Lemma}
\newtheorem{proposition}[theorem]{Proposition}
\newtheorem{corollary}[theorem]{Corollary}
\theoremstyle{definition}
\newtheorem{definition}[theorem]{Definition}
\newtheorem{example}[theorem]{Example}

\theoremstyle{remark}
\newtheorem{remark}[theorem]{Remark}
\numberwithin{equation}{section}

\newcommand{\la}{\langle}
\newcommand{\ra}{\rangle}

\begin{document}
\setcounter{page}{1}

\title[Finite extensions of Bessel sequences]{Finite extensions of Bessel sequences}

\author[D. Baki\' c, T. Beri\' c]{Damir Baki\' c$^1$ and Tomislav Beri\' c$^1$}

\address{$^{1}$ Department of Mathematics, University of Zagreb, Bijeni\v{c}ka 30, 10000 Zagreb, Croatia.}
\email{\textcolor[rgb]{0.00,0.00,0.84}{bakic@math.hr;
tberic@math.hr}}

\subjclass[2010]{Primary 42C15; Secondary 47A05.}

\keywords{Bessel sequences, frames, Parseval frames.}

\date{Received: xxxxxx; Revised: yyyyyy; Accepted: zzzzzz.}

\begin{abstract}
The paper studies finite extensions of Bessel sequences in infinite-dimensional Hilbert spaces. We provide a characterization of Bessel sequences that can be extended to frames by adding finitely many vectors. We also characterize frames that can be converted to Parseval frames by finite-dimensional perturbations. Finally, some results on excesses of frames and near-Riesz bases are derived.
\end{abstract} \maketitle

\section{Introduction and preliminaries}

A sequence $(f_n)_{n=1}^{\infty}$ in a Hilbert space $H$ is a frame for $H$ if there exist positive constants $A$ and $B$, that are called frame bounds, such that
\begin{equation}\label{frame def}
A\|x\|^2\leq \sum_{n=1}^{\infty}|\la x,f_n\ra |^2\leq B\|x\|^2,\,\forall x \in H.
\end{equation}
Frame bounds are not unique. The optimal upper frame bound is the infimum over all upper frame bounds, and the optimal lower frame bound is the supremum over all lower frame bounds.  If $A=B$ we say that frame is tight and, in particular, if $A=B=1$ so that
\begin{equation}
\sum_{n=1}^{\infty}|\la x,f_n\ra |^2=\|x\|^2,\,\forall x \in H,
\end{equation}
we say that $(f_n)_{n=1}^{\infty}$ is a Parseval frame.

A sequence $(f_n)_{n=1}^{\infty}$ in $H$ is called a Bessel sequence if only the right hand inequality in (\ref{frame def}) is satisfied. In that case $B$ is called a Bessel bound. The optimal Bessel bound is the infimum over all Bessel bounds.

Frames were first introduced by Duffin and Schaeffer (\cite{DS}). Today frames play important roles in many applications in mathematics, science and engineering. For general information and basic facts about frames we refer the reader to \cite{Cas}, \cite{Chr}, \cite{Dau} and \cite{HW}.

Extensions principles of Bessel sequences to frames are known from the literature in frame theory. In \cite{CL} Casazza and Leonhard showed that any Bessel sequence in a finite-dimensional space can be extended to a tight frame. Later on, this result was extended to the infinite-dimensional case by Li and Sun in \cite{deng}. In \cite{CKK} Christensen, Kim and Kim showed that in any separable Hilbert space each pair of Bessel sequences can be extended to a pair of mutually dual frames.

In this paper we study {\em finite} extensions of Bessel sequences and frames in infinite-dimensional Hilbert spaces. Our main result is Theorem \ref{cetiri}  (see also Proposition \ref{jedan} and Theorem \ref{prvi glavni}) which gives a necessary and sufficient condition under which a Bessel sequence can be extended to a frame by adding a finite sequence of vectors. The characterizing condition (which is named essential duality of Bessel sequences) is concerned with pairs of Bessel sequences: the product of the synthesis operator of one sequence by the analysis operator of the other one is the unit element in the Calkin algebra of the underlying Hilbert space.


In the same spirit, Theorem \ref{drugi glavni} provides a necessary and sufficient condition for the existence of a finite extension of a Bessel sequence to a
Parseval frame - a result that in a similar form can be found in \cite{deng}.

Analyzing these conditions we then provide in Theorem \ref{treci glavni} a characterization of frames
for which there exists a finite-dimensional perturbation that is a Parseval frame.

The concluding Section 3 contains some results on excesses of frames. Corollary \ref{sedmi glavni} provides a new characterization of frames with finite excess among all frames.
It turns out that the characterizing condition is a mirror image of the "essential duality" of Bessel sequences introduced in Section 2: the only difference is the order of the action of the analysis and the synthesis operators involved.
This and other results in Section 3 are derived as simple consequences of our results from the preceding section.

\vspace{.1in}

In the rest of this introductory section we summarize some basic facts on Bessel sequences and frames that will be needed in the sequel and fix our notation.

Throughout, $H$ will denote a complex, infinite-dimensional separable Hil\-bert space. By  $\Bbb B(H,K)$ we denote the space of all bounded operators between Hilbert spaces $H$ and $K$. For $H=K$ we write $\Bbb B(H)$.

The identity operator is denoted by $I$. The range and the kernel of an operator $T\in \Bbb B(H,K)$ are denoted by $\mbox{Im}\,T$ and $\mbox{Ker}\,T$, respectively.   If $T\in \Bbb B(H,K)$ has closed range, its pseudo-inverse is denoted by $T^{\dag}$. For basic facts concerning pseudo-inverses we refer to appendix A7 in \cite{Chr}.

For each Bessel sequence $(f_n)_{n=1}^{\infty}$ in $H$  one defines the analysis operator $U:H \rightarrow l^2$ by $Ux=(\la x,f_n\ra)_n$. It is well known that $U$ is a bounded linear operator. Moreover, it turns out that its adjoint $U^*$, that is called the synthesis operator, satisfies $U^*((\alpha_n))=\sum_{n=1}^{\infty} \alpha_nf_n$ and, in particular,
$U^*e_n=f_n,\,\forall n \in \Bbb N$, where $(e_n)$ is the canonical orthonormal basis for $l^2$.

Note that each finite sequence is obviously Bessel. Here we adopt the following convention:
even if $(x_n)_{n=1}^{k}$ is a finite sequence in $H$, we will understand that its analysis operator takes values in $l^2$ (in other words, we will tacitly assume that $x_1,\ldots,x_k$ are followed by infinitely many zeros).

If $(f_n)_{n=1}^{\infty}$ is a frame for $H$, the analysis operator is not only bounded, but also bounded from below. Moreover, $U^*U$ is then an invertible operator (called the frame operator) on $H$ which satisfies $A\cdot I \leq U^*U \leq B \cdot I$ and $\|U^*U\|\leq B$. If $B$ is the optimal upper frame bound we have $\|U^*U\|= B$. Clearly, $(f_n)_{n=1}^{\infty}$ is a Parseval frame if and only if $U$ is an isometry, {\em i.e.} if and only if $U^*U=I$.

For each frame $(f_n)_{n=1}^{\infty}$ its canonical dual is the frame $(\tilde{f_n})_{n=1}^{\infty}$ defined by $\tilde{f_n}=(U^*U)^{-1}f_n,\,n\in \Bbb N$. It is one of the basic facts of the theory of frames that $(\tilde{f_n})_{n=1}^{\infty}$ (as any other dual) enables complete reconstruction in terms of the original frame $(f_n)_{n=1}^{\infty}$:
\begin{equation}\label{dual frame reconstruction}
x=\sum_{n=1}^{\infty}\la x,\tilde{f_n} \ra f_n,\,\forall x \in H.
\end{equation}

Also, it is well known that for each frame $(f_n)_{n=1}^{\infty}$ the sequence $(\overline{f_n})_{n=1}^{\infty}$, $\overline{f_n}=(U^*U)^{-\frac{1}{2}}f_n,\,n\in \Bbb N$, is a Parseval frame.

If $(f_n)_{n=1}^{\infty}$ is a Parseval frame, then $f_n=\tilde{f_n}=\overline{f_n},\,\forall n \in \Bbb N$.

\section{Finite extensions to frames}

Suppose we have a Bessel sequence $(f_n)_{n=1}^{\infty}$ in $H$ and consider the following question: does there exist a sequence $(x_n)_{n}$ in $H$ such that the extended sequence $(x_n)_n \cup (f_n)_{n=1}^{\infty}$ is a frame for $H$?

Note that the answer is trivial if we are allowed to add an infinite sequence: any Bessel system can be extended to a frame by adding in a frame for $H$. Thus, the nontrivial part of our question consists of characterizing those Bessel sequences that admit finite extensions to frames.

\vspace{.1in}

First observe: if $(f_n)_{n=1}^{\infty}$ admits a finite extension to a frame for $H$ then, obviously, the deficit of $(f_n)_{n=1}^{\infty}$ must be finite.

Recall from \cite{CHet_al} that the deficit of a sequence $(f_n)_{n=1}^{\infty}$ in $H$ is defined as the least cardinal $d((f_n)_{n=1}^{\infty})$ such that there exists a subset $G$ of $H$ of cardinality $d((f_n)_{n=1}^{\infty})$ so that $\overline{\mbox{span}}\,((f_n)_{n=1}^{\infty} \cup G)=H$. If $(f_n)_{n=1}^{\infty}$ is a Bessel sequence in $H$ with the analysis operator $U$, one easily proves (see Lemma 4.1. in \cite{CHet_al}) that $d((f_n)_{n=1}^{\infty})=\mbox{dim}(\mbox{Ker}\,U)$.

So, $d((f_n)_{n=1}^{\infty})<\infty$ is a necessary condition if we want to obtain a frame from a Bessel sequence by adding only finitely many vectors. However, this is not enough.


\begin{example}
Consider the canonical orthonormal basis $(e_n)$ for $l^2$ and the sequence $(f_n)_{n=1}^{\infty}$ defined by
$f_1=e_1$, $f_n=e_{n-1}+e_n,\,n\geq 2$.
Clearly, $(f_n)_{n=1}^{\infty}$ is a Bessel sequence in $l^2$ with the analysis operator $U=S+I$, where $S$ is the unilateral shift on $l^2$. Since $S$ has no eigenvalues, we have $\mbox{dim}(\mbox{Ker}\,U)=0$; hence, $d((f_n)_{n=1}^{\infty})=0$.

However, one can not extend $(f_n)_{n=1}^{\infty}$ to a frame by adding finitely many vectors. This can be seen directly (we omit the details), but also, since $\mbox{Im}\,(S+I)$ is not a closed subspace of $l^2$, by applying our Proposition \ref{jedan} and Lemma \ref{tri} below.
\end{example}

\vspace{.1in}

In order to characterize all Bessel sequences which admit finite extensions to frames we first provide another necessary condition.

\begin{proposition}\label{jedan}
Let $(f_n)_{n=1}^{\infty}$ be a Bessel sequence in $H$ for which there exists a finite sequence $(x_n)_{n=1}^{k}$ in $H$ such that the extended sequence $(x_n)_{n=1}^{k} \cup (f_n)_{n=1}^{\infty}$ is a frame for $H$. Then there exists a Bessel sequence $(g_n)_{n=1}^{\infty}$ in $H$ such that the operator $I-V^*U$ has finite rank, where $U$ and $V$ denote the analysis operators of $(f_n)_{n=1}^{\infty}$ and $(g_n)_{n=1}^{\infty}$, respectively.
\end{proposition}
\begin{proof}
Let us take any dual frame of $(x_n)_{n=1}^{k} \cup (f_n)_{n=1}^{\infty}$ and denote it, for convenience, by $(y_n)_{n=1}^{k} \cup (g_n)_{n=1}^{\infty}$ (in other words, the first $k$ elements $y_1,\ldots , y_k$ are followed by $g_1,g_2,\ldots$). Let $V_1$ be its analysis operator.
Now observe that $(g_n)_{n=1}^{\infty}$ is a Bessel sequence; let us denote its analysis operator by $V$. Finally, let $\tilde{V}$ be the analysis operator of the Bessel sequence $0,\ldots , 0,g_1,g_2,\ldots$ ($k$ zeros). Then the operator $Q=V_1-\tilde{V}$ has finite rank.

Further, let us denote by $U_1$ the analysis operator of the frame $(x_n)_{n=1}^{k} \cup (f_n)_{n=1}^{\infty}$.
Since $(y_n)_{n=1}^{k} \cup (g_n)_{n=1}^{\infty}$ is a dual of $(x_n)_{n=1}^{k} \cup (f_n)_{n=1}^{\infty}$, we have
\begin{equation}\label{nulta dodatna}
V_1^*U_1=I.
\end{equation}
On the other hand, $(\tilde{V})^*U_1x=\sum_{n=1}^{\infty}\la x, f_n\ra g_n=V^*Ux,\,\forall x \in \Bbb H$; thus,
\begin{equation}\label{prva dodatna}
(\tilde{V})^*U_1=V^*U.
\end{equation}
From this we find
$$
I-V^*U\stackrel{(\ref{prva dodatna})}{=}I-(\tilde{V})^*U_1=I-(V_1-Q)^*U_1=I-V_1^*U_1+Q^*U_1\stackrel{(\ref{nulta dodatna})}{=}Q^*U_1.
$$
Hence,
$I-V^*U=Q^*U_1$ is a finite rank operator. 
\end{proof}

\vspace{.1in}

It will turn out that the converse of the preceding proposition is also true. In fact, we shall prove the converse in a stronger form. To do that we first need a lemma on Hilbert space operators known as a part of Atkinson's theorem. For the proof, which is omitted, we refer the reader to the solution to Problem 181 in \cite{H}.

\begin{lemma}\label{tri}
Let $H$ and $K$ be Hilbert spaces and $U\in \Bbb B(H,K)$. Suppose that there exists $V\in \Bbb B(H,K)$ such that the operator $I-V^*U$ is compact. Then $\mbox{Im}\,U$ is a closed subspace of $K$, $\mbox{dim}\,(\mbox{Ker}\,U)<\infty$ and $U^{\dag}U=I-P$, where $P$ is the orthogonal projection to $\mbox{Ker}\,U$.
\end{lemma}

\begin{theorem}\label{prvi glavni}
Let $(f_n)_{n=1}^{\infty}$ and $(g_n)_{n=1}^{\infty}$ be Bessel sequences in $H$ with the optimal Bessel bounds $B$ and $D$ and the analysis operators $U$ and $V$, respectively. Suppose that $I-V^*U$ is a compact operator. Then there exist finite sequences $(x_n)_{n=1}^{k}$ and $(y_n)_{n=1}^{l}$ such that $(x_n)_{n=1}^{k} \cup (f_n)_{n=1}^{\infty}$ and $(y_n)_{n=1}^{l} \cup (g_n)_{n=1}^{\infty}$ are frames for $H$ with the optimal upper frame bounds $B$ resp.~$D$.
\end{theorem}
\begin{proof}
By the preceding lemma we have $\mbox{dim}(\mbox{Ker}\,U)<\infty$, so one can find a finite frame $x_1,\ldots , x_k$ for $\mbox{Ker}\,U$
with the optimal upper frame bound $B$. Let us denote by $F$ the corresponding analysis operator.
Take any dual frame $z_1,\ldots ,z_k$ for $x_1,\ldots , x_k$ with the analysis operator $G$. We assume that all $x_j$'s and $z_j$'s belong to $\mbox{Ker}\,U$. Then we have $G^*F=P$, where $P$ denotes the orthogonal projection to $\mbox{Ker}\,U$. Note also that both $\mbox{Im}\,F$ and $\mbox{Im}\,G$ are contained in $M_k=\mbox{span}\,\{e_1,\ldots , e_k\}$, where $(e_n)$ denotes the canonical orthonormal basis for $l^2$.

Let us now take the extended sequence $x_1,\ldots , x_k,f_1,f_2,\ldots$. Observe that its analysis operator $U_1$ is given by $U_1=F+S^kU$, where $S$ denotes the unilateral shift on $l^2$. Let $W=G+S^k(U^{\dag})^*$. Then, using the equalities $F^*S^k=G^*S^k=0$ and the last assertion of the preceding lemma, we obtain $$W^*U_1=(G^*+U^{\dag}(S^*)^k)(F+S^kU)=G^*F+U^{\dag}U=P+(I-P)=I.$$ This implies that $U_1$ is bounded from below; thus, $(x_n)_{n=1}^{k} \cup (f_n)_{n=1}^{\infty}$ is a frame.

For $x\in \mbox{Ker}\,U$ we have $0=\|Ux\|^2=\sum_{n=1}^{\infty}|\la x,f_n\ra|^2$ and $\sum_{n=1}^{k}|\la x,x_n\ra|^2\leq B\|x\|^2$. On the other hand, if $x \in (\mbox{Ker}\,U)^{\perp}$ then $\sum_{n=1}^{k}|\la x,x_n\ra|^2=0$.

Let us now take an arbitrary $x \in H$ and write $x=a+b$ with $a\in \mbox{Ker}\,U$ and $b\in (\mbox{Ker}\,U)^{\perp}$. Then
$$\sum_{n=1}^{k}|\la x,x_n\ra|^2+\sum_{n=1}^{\infty}|\la x,f_n\ra|^2=\sum_{n=1}^{k}|\la a,x_n\ra|^2+\|U(a+b)\|^2$$
$$\leq B\|a\|^2+\|Ub\|^2\leq B(\|a\|^2+\|b\|^2)=B\|x\|^2.$$

The assertions concerning $(g_n)_{n=1}^{\infty}$ now follow by the same arguments using compactness of the operator $I-U^*V$.
\end{proof}

\begin{remark}
As the preceding proof shows, the extension of a Bessel sequence to a frame is not unique, even if we insist (as we did) on the same upper frame bound. It is also clear that the minimal number of elements that should be added to a given Bessel sequence $(f_n)_{n=1}^{\infty}$ in order to obtain a frame is $d= \mbox{dim}\,(\mbox{Ker}\,U)$. In that sense a minimal choice is $(\sqrt{B}w_1,\ldots , \sqrt{B}w_d)$, where $(w_1,\ldots ,w_d)$ is an orthonormal basis for $\mbox{Ker}\,U$.
\end{remark}

\vspace{.1in}

Proposition \ref{jedan} and Theorem \ref{prvi glavni} motivate the following definition:

\begin{definition}\label{ess dual}
We say that Bessel sequences $(f_n)_{n=1}^{\infty}$ and $(g_n)_{n=1}^{\infty}$ with the analysis operators $U$ and $V$ are essentially dual to each other if $I-V^*U$ is a compact operator.
\end{definition}

If $I-V^*U$ is compact then, obviously, $I-U^*V$ is compact as well; hence essential duality of Bessel sequences is a symmetric relation.

Now we can summarize the statements of Proposition \ref{jedan} and Theorem \ref{prvi glavni} in the following simple way:

\begin{theorem}\label{cetiri}
A Bessel sequence $(f_n)_{n=1}^{\infty}$ has  a finite extension to a frame if and only if there exists a Bessel sequence essentially dual to $(f_n)_{n=1}^{\infty}$.
\end{theorem}

\vspace{.1in}

Next we discuss finite extensions of Bessel sequences to Parseval frames. Again, we shall first obtain necessary conditions.

Suppose we have a Bessel sequence $(f_n)_{n=1}^{\infty}$ in $H$ for which there exists a finite sequence $(x_n)_{n=1}^{k}$ such that
$(x_n)_{n=1}^{k} \cup (f_n)_{n=1}^{\infty}$ is a Parseval frame for $H$. Denote by $B$ the optimal Bessel bound of $(f_n)_{n=1}^{\infty}$.
Since
$$\sum_{n=1}^{\infty}|\la x,f_n\ra|^2\leq \sum_{n=1}^{k}|\la x,x_n\ra|^2+\sum_{n=1}^{\infty}|\la x,f_n\ra|^2=\|x\|^2,\,\forall x \in H,$$
we conclude that $B\leq 1$.

Let $U$ be the analysis operator of $(f_n)_{n=1}^{\infty}$. Denote by $F$ the analysis operator of $(x_n)_{n=1}^{k}$; since this sequence is finite, $F$ is a finite rank operator. Now observe that the analysis operator $U_1$ of the sequence $(x_n)_{n=1}^{k} \cup (f_n)_{n=1}^{\infty}$ is given by $U_1=F+S^kU$, where, as before, $S$ denotes the unilateral shift on $l^2$.

Since by our assumption $(x_n)_{n=1}^{k} \cup (f_n)_{n=1}^{\infty}$ is a Parseval frame for $H$, we have $U_1^* U_1 = I$. From this we get
$$I=(F+S^kU)^*(F+S^kU)=F^*F+F^*S^kU+U^*(S^k)^*F+U^*U.$$
Let $K=F^*F+F^*S^kU+U^*(S^k)^*F$. Then $K$ is a finite rank operator and $I-U^*U=K$.
In particular, the operator $I-U^*U$ is not invertible because it has finite-dimensional range. This in turn implies that $1$ belongs to the spectrum of $U^*U$ and this, together with our previous observation $B\leq 1$, implies $B=1$.

The statement of the following theorem appears in a similar form in \cite{deng}. Although the proof in \cite{deng} uses g-frames, the key argument is essentially the same as in our proof below. 

\begin{theorem}\label{drugi glavni}
Let $(f_n)_{n=1}^{\infty}$ be a Bessel sequence in $H$ with the optimal Bessel bound $B$ and the analysis operator $U$. The following conditions are mutually equivalent:
\begin{itemize}
\item[(a)] There exists a finite sequence $(x_n)_{n=1}^{k}$ in $H$ such that $(x_n)_{n=1}^{k} \cup (f_n)_{n=1}^{\infty}$ is a Parseval frame for $H$.
\item[(b)] $B= 1$ and $\mbox{dim}\,(\mbox{Im}(I-U^*U))<\infty$.
\end{itemize}
\end{theorem}
\begin{proof}
(a) $\Rightarrow$ (b) is already proved in the preceding discussion.
Let us prove \mbox{(b) $\Rightarrow$ (a)}.

Since $B= 1$, the square root $(I-U^*U)^{\frac{1}{2}}$ is a well defined and positive operator.
Observe that $\mbox{Ker}\,(I-U^*U)^{\frac{1}{2}}= \mbox{Ker}\,(I-U^*U)$. By taking orthogonal complements we get $\mbox{Im}\,(I-U^*U)^{\frac{1}{2}}= \mbox{Im}\,(I-U^*U)$ (since, by assumption, this subspace is finite-dimensional, the closure signs are superfluous).

Let $k=\mbox{dim}\,(\mbox{Im}\,(I-U^*U))$ and $M_k=\mbox{span}\,\{e_1,\ldots,e_k\}\leq l^2$. Take any partial isometry $R\in \Bbb B(H,l^2)$  with the initial subspace $\mbox{Im}\,(I-U^*U)$ and the final subspace $M_k$.
Notice that $\mbox{Im}\,R \perp \mbox{Im}\,S^kU$.

Let $U_1=R(I-U^*U)^{\frac{1}{2}}+S^kU$. We claim that $U_1$ is an isometry. Indeed, for $x \in H$ we have
$$\|U_1x\|^2=\|R(I-U^*U)^{\frac{1}{2}}x+S^kUx\|^2=\|R(I-U^*U)^{\frac{1}{2}}x\|^2+\|S^kUx\|^2$$
$$=
\|(I-U^*U)^{\frac{1}{2}}x\|^2+\|Ux\|^2=\la (I-U^*U)x,x\ra + \la U^*Ux,x\ra=\|x\|^2.$$

Since $U_1$ is an isometry, $(U_1^*e_n)_{n=1}^{\infty}$ is a Parseval frame for $H$. Observe that $U_1^*=(I-U^*U)^{\frac{1}{2}}R^*+U^*(S^*)^k$
implies $U_1^*e_{k+j}=U^*e_j=f_j,\,\forall j\in \Bbb N$. Thus, our original Bessel sequence $(f_n)_{n=1}^{\infty}$ is extended to a Parseval frame by the elements $x_j=(I-U^*U)^{\frac{1}{2}}R^*e_j,\,j=1,2,\ldots,k$.
\end{proof}

\begin{remark}\label{nejedinstvenost}
Suppose that $l\geq k=\mbox{dim}\,(\mbox{Im}\,(I-U^*U))$ and take a partial isometry $R^{\prime}$ with the initial subspace $\mbox{Im}\,(I-U^*U)$ and the final subspace {\em contained} in $M_l=\mbox{span}\,\{e_1,\ldots,e_l\}\leq l^2$. Then the same argument as above applies if we replace $U_1$ by $U_1^{\prime}=R^{\prime}(I-U^*U)^{\frac{1}{2}}+S^lU$.
This would result with an extension of the original Bessel sequence to a Parseval frame by adding $l$ elements.

If we assume that $(y_n)_{n = 1}^{k^{\prime}} \cup (f_n)_{n = 1}^{\infty}$ is a Parseval frame for $k^{\prime} < k$, then for each $x \in H$ we have that $x = \sum_{n = 1}^{\infty} \la x, f_n \ra f_n + \sum_{n = 1}^{k^{\prime}} \la x, y_n \ra y_n$. Then
$$
\mbox{Im} \left( I - U^* U \right) = \left\{ x - \sum_{n = 1}^{\infty} \la x, f_n \ra f_n : x \in H \right\} = \left\{ \sum_{n = 1}^{k^{\prime}} \la x, y_n \ra y_n : x \in H \right\}.
$$
Therefore $k = \dim \mbox{Im} \left( I - U^* U \right) \le \dim \mbox{span} \{ y_n \}_{n = 1}^{k^{\prime}} \le k^{\prime} < k$ which is a contradiction. Thus, the minimal number of elements that one must add to a given Bessel sequence in order to obtain a Parseval frame is $k=\mbox{dim}\,(\mbox{Im}\,(I-U^*U))$. Such minimal extensions are described in Corollary \ref{minimal}. Here we also mention the concluding Corollary \ref{zadnji} which provides another related result.
\end{remark}

\begin{remark}
We also note that for the proof of (b) $\Rightarrow$ (a) in the preceding theorem it suffices to assume
$B \leq 1$ and $\mbox{dim}\,(\mbox{Im}(I-U^*U))<\infty$.
\end{remark}

\begin{corollary}\label{minimal}
Let $(f_n)_{n=1}^{\infty}$ be a Bessel sequence in $H$ with a Bessel bound less than or equal to $1$ and such that
$\mbox{dim}\,(\mbox{Im}(I-U^*U))=k<\infty$. Let $x_j=(I-U^*U)^{\frac{1}{2}}w_j,\,j=1,\ldots , k$, where $(w_1,\ldots , w_k)$ is an orthonormal basis for
$\mbox{Im}\,(I-U^*U)$. Then $(x_n)_{n=1}^{k} \cup (f_n)_{n=1}^{\infty}$ is a Parseval frame for $H$.
\end{corollary}

\vspace{.1in}

When one compares the statement of Theorem \ref{drugi glavni} with those of Theorem \ref{prvi glavni} and Corollary \ref{cetiri} it is natural to ask the following question: is it enough, in order to ensure a finite extension of a given Bessel sequence to a Parseval frame, to assume that $I-U^*U$ is only a compact operator (together with $B\leq 1$)?

The answer is negative. Namely, if $I-U^*U$ is a compact operator, then there does exist a bounded operator $V$ such that $I-V^*U$ has finite rank, but one can not conclude that the rank of $I-U^*U$ is finite.

Here is an example. Consider the canonical orthonormal basis $(e_n)$ for $l^2$ and the sequence $(f_n)_{n=1}^{\infty}$ defined by $f_n=\sqrt{\frac{n}{n+1}}e_n,\,n \in \Bbb N$. Clearly, $(f_n)_{n=1}^{\infty}$ is a frame; in fact, a Riesz basis with the upper frame bound $B=1$. If we denote by $U$ its analysis operator, then $U^*Ux=\sum_{n=1}^{\infty}\frac{n}{n+1}\la x,e_n\ra e_n,\,\forall x \in l^2$. This implies $(I-U^*U)x=
\sum_{n=1}^{\infty}\frac{1}{n+1}\la x,e_n\ra e_n,\,\forall x \in l^2$; thus, $I-U^*U$ is a compact operator. However; the sequence $(f_n)_{n=1}^{\infty}$ can not be extended to a Parseval frame by adding a finite number of elements. This can be seen directly, but the same conclusion also follows from Theorem \ref{drugi glavni}: namely, it is evident that the operator $I-U^*U$ has infinite rank.

\vspace{.1in}

We end this discussion with a corollary on extensions to tight frames of Bessel sequences with an arbitrary Bessel bound.

\begin{corollary} \label{pet}
Let $(f_n)_{n=1}^{\infty}$ be a Bessel sequence in $H$ with the optimal Bessel bound $B$ and the analysis operator $U$. The following conditions are mutually equivalent:
\begin{itemize}
\item[(a)] There exists a finite sequence $(x_n)_{n=1}^{k}$ in $H$ such that $(x_n)_{n=1}^{k} \cup (f_n)_{n=1}^{\infty}$ is a $B$-tight frame for $H$.
\item[(b)] $\mbox{dim}\,(\mbox{Im}(B\cdot I-U^*U))<\infty$.
\end{itemize}
\end{corollary}
\begin{proof}
Consider the Bessel sequence $(g_n)_{n=1}^{\infty}$, where $g_n=\frac{1}{\sqrt{B}}f_n,\, n \in \Bbb N$, and observe that its optimal Bessel bound is equal to $1$, and its analysis operator is $\frac{1}{\sqrt{B}}U$. Now Theorem \ref{drugi glavni} applies. 
\end{proof}

\vspace{.1in}

Theorem \ref{drugi glavni} shows that Bessel sequences and, in particular, frames for which $I-U^*U$ is a finite rank operator are, in a sense, almost Parseval. Our next theorem provides two more characterizing properties of such frames which also show, in a different way, a close relation with the class of Parseval frames.

\begin{theorem}\label{treci glavni}
Let $(f_n)_{n=1}^{\infty}$ be a frame for $H$ with the analysis operator $U$. The following conditions are mutually equivalent:
\begin{itemize}
\item[(a)] $\mbox{dim}\,(\mbox{Im}(I-U^*U))<\infty$.
\item[(b)] There exists a sequence $(g_n)_{n=1}^{\infty}$ in a finite-dimensional subspace $L$ of $H$ such that $(f_n+g_n)_{n=1}^{\infty}$ is a Parseval frame for $H$.
\item[(c)] $x=\sum_{n=1}^{\infty}\la x,f_n\ra f_n,\,\forall x \in M$, where $M$ is a closed subspace of $H$ of finite co-dimension.
\end{itemize}
\end{theorem}
\begin{proof}
(a) $\Rightarrow$ (b).

Since $I-U^*U$ is a selfadjoint operator and $\mbox{dim}\,(\mbox{Im}(I-U^*U))<\infty$, we have $H=\mbox{Ker}(I-U^*U)\oplus \mbox{Im}(I-U^*U)$. This implies $\mbox{Im}\,U=U(\mbox{Ker}(I-U^*U))+U(\mbox{Im}(I-U^*U))$. Moreover, we claim that this sum is direct. Indeed, if $Ux=Uy$ for some
$x \in \mbox{Ker}(I-U^*U)$ and $y \in \mbox{Im}(I-U^*U)$ then, by injectivity of $U$, we conclude $x=y$ and hence $x=y=0$.

Now observe that all direct complements of a closed subspace in a given space are of the same dimension. Thus,
$$
\mbox{dim}\,(\mbox{Im}\,U \ominus U(\mbox{Ker}(I-U^*U)))=\mbox{dim}\,(U(\mbox{Im}(I-U^*U)))=\mbox{dim}\,(\mbox{Im}(I-U^*U))
$$
where the last inequality follows from injectivity of $U$.
This allows us to find an isometry $F_0:\mbox{Im}(I-U^*U) \rightarrow \mbox{Im}\,U$ such that $\mbox{Im}\,F_0 \perp U(\mbox{Ker}(I-U^*U))$.
Put $U_0=U|_{\mbox{Ker}(I-U^*U)}$. Since we have $U^*Ux=x$ for $x \in \mbox{Ker}(I-U^*U)$, $U_0$ is also an isometry. Finally, let $V=U_0\oplus F_0$; since the images of $U_0$ and $F_0$ are mutually orthogonal, $V$ is an isometry. So, if we define $h_n=V^*e_n,\,n\in \Bbb N$, where $(e_n)$ is the canonical orthonormal basis for $l^2$, the sequence $(h_n)_{n=1}^{\infty}$ is a Parseval frame for $H$. Obviously, $F=V-U$ is a finite rank operator. Namely, $F$ acts trivially on $\mbox{Ker}(I-U^*U)$ and $\mbox{dim}\,(\mbox{Im}(I-U^*U))<\infty$. Put $g_n=F^*e_n,\,n \in \Bbb N$. Then
$(g_n)_{n=1}^{\infty}$ is a finite-dimensional perturbation of $(f_n)_{n=1}^{\infty}$ such that $(h_n=f_n+g_n)_{n=1}^{\infty}$ is a Parseval frame.

(b) $\Rightarrow$ (a).

If we assume (b) then $(g_n)_{n=1}^{\infty}$ is a Bessel sequence, as the difference of two Bessel sequences. Let $T$ denotes its analysis operator. By assumption, $T^*$ is a finite rank operator. Hence, $T$ has a finite rank as well. Since $(U^*+T^*)e_n=f_n+g_n,\,\forall n \in \Bbb N$, and
$(f_n+g_n)_{n=1}^{\infty}$ is a Parseval frame for $H$, $U+T$ is an isometry. Thus, $I=(U+T)^*(U+T)=U^*U+U^*T+T^*U+T^*T$. Let $F=U^*T+T^*U+T^*T$. Then $F$ is a finite rank operator and we have $I-U^*U=F$.

(a) $\Rightarrow$ (c).

Let $M=\mbox{Ker}(I-U^*U)$. By assumption, $M^{\perp}=\mbox{Im}(I-U^*U)$ is finite-dimensional. For $x \in M$ we have $x=U^*Ux$ {\em i.e.}
$x=\sum_{n=1}^{\infty}\la x,f_n\ra f_n$.

(c) $\Rightarrow$ (a).

Suppose we have $x=\sum_{n=1}^{\infty}\la x,f_n\ra f_n$ for all $x \in M$ with $\mbox{dim}\,M^{\perp}<\infty$. This implies $U^*Ux=x,\,\forall x \in M$. Hence, $M \subseteq \mbox{Ker}(I-U^*U)$, and this implies $\mbox{Im}(I-U^*U) \subseteq M^{\perp}$.
\end{proof}

\begin{remark}
(i) In contrast to Theorem \ref{drugi glavni}, a general assumption in the preceding theorem is that  $(f_n)_{n=1}^{\infty}$ is a frame, not merely a Bessel sequence. The reason for that is the proof of the above implication (a) $\Rightarrow$ (b) where we have used injectivity of the analysis operator $U$.

(ii) Note that in the property (c) we do not claim that the frame elements $f_n$ belong to $M$. In this light, we may say that  $(f_n)_{n=1}^{\infty}$ is, in a sense, an outer frame for the subspace $M$.
\end{remark}

\begin{remark}
Let us show an alternative proof of (a) $\Rightarrow$ (b) from the preceding theorem.

 Consider the decomposition $H=\mbox{Ker}(I-U^*U) \oplus \mbox{Im}(I-U^*U)$. Since $\mbox{Ker}(I-U^*U)$ is an eigenspace for the operator $U^*U$, its orthogonal complement is invariant for $U^*U$. So, both subspaces in the above decomposition are invariant for $U^*U$ and we may write
$U^*U=I\oplus \left[\begin{array}{ccc}\lambda_1& & \\
 &\ddots& \\
  & &\lambda_k\end{array}\right]$ with respect to an appropriately chosen orthonormal basis $(w_1,\ldots , w_k)$ for $\mbox{Im}(I-U^*U)$. Note that the optimal frame bounds of $(f_n)_{n = 1}^{\infty}$ satisfy $A \leq \lambda_i \le B, i = 1, \ldots, k$. Let us write, with respect to the same decomposition of $H$, $f_n=x_n+y_n,\,n \in \Bbb N$.

Now observe that both subspaces are also invariant for $(U^*U)^{-\frac{1}{2}}$ and that $(U^*U)^{-\frac{1}{2}}$ acts as the identity operator on $\mbox{Ker}(I-U^*U)$.

Put $\overline{f_n}=(U^*U)^{-\frac{1}{2}}f_n,\,n\in \Bbb N$. Recall that  $(\overline{f_n})_{n=1}^{\infty}$ is the Parseval frame canonically associated with $(f_n)_{n=1}^{\infty}$.
Obviously, for all $n\in \Bbb N$ we have
$\overline{f_n}=x_n+z_n$, where $z_n=\left[\begin{array}{ccc}\lambda_1^{-\frac{1}{2}}& & \\
 &\ddots& \\
  & &\lambda_k^{-\frac{1}{2}}\end{array}\right]y_n$. The sequences $(y_n)_{n=1}^{\infty}$ and $(z_n)_{n=1}^{\infty}$ belong to a finite-dimensional subspace $\mbox{Im}(I-U^*U)$, so does their difference and we have $f_n+(z_n-y_n)=x_n+y_n+z_n-y_n=\overline{f_n},\,\forall n \in \Bbb N$.
\end{remark}

\vspace{.2in}

\section{Frames with finite excess}

Let $(f_n)_{n=1}^{\infty}$ be a frame for $H$. The excess $e((f_n)_{n=1}^{\infty})$ is defined as the maximal number of elements that can be deleted from $(f_n)_{n=1}^{\infty}$ such that the reduced sequence remains a frame. If we denote by $U$ the analysis operator of $(f_n)_{n=1}^{\infty}$, it is well known that $e((f_n)_{n=1}^{\infty})=\mbox{dim}\,(\mbox{Ker}\,U^*)$. If we denote by $Q\in \Bbb B(l^2)$ the orthogonal projection to
$\mbox{Ker}\,U^*$, then $e((f_n)_{n=1}^{\infty})=\mbox{tr}(Q)$. For these and other facts concerning excess we refer the reader to \cite{CHet_al}.

When $e((f_n)_{n=1}^{\infty})$ is finite, we say that $(f_n)_{n=1}^{\infty}$ is a near-Riesz basis. It is proved in \cite{Holub} that a frame $(f_n)_{n=1}^{\infty}$ for $H$ is a near-Riesz basis if and only if there exists a finite set $J\subseteq \Bbb N$ such that $(f_n)_{n\in \Bbb N \setminus J}$ is a Riesz basis for $H$.

A direct application of Lemma \ref{tri} gives us a new characterization of near-Riesz bases.

\begin{theorem}\label{peti glavni}
Let $(f_n)_{n=1}^{\infty}$ be a frame for $H$ with the analysis operator $U$. Then $(f_n)_{n=1}^{\infty}$ is a near-Riesz basis if there exists an operator $V\in \Bbb B(H,l^2)$ such that $I-VU^*$ is a compact operator.
\end{theorem}
\begin{proof}
By Lemma \ref{tri}, $\mbox{Ker}\,U^*$ is a finite-dimensional subspace od $l^2$.
\end{proof}

\vspace{.1in}

Since there is a bijective correspondence between bounded operators in $\Bbb B(H,l^2)$ and Bessel sequences in $H$, one can restate the preceding theorem in terms of the existence of a Bessel sequence $(g_n)_{n=1}^{\infty}$ with the analysis operator $V$ such that $I-VU^*$ is compact. By taking adjoints we conclude that $I-UV^*$ is also compact, and it follows that $\mbox{Ker}\,V^*$ is a finite-dimensional subspace od $l^2$. However, this is not enough to conclude that $e((g_n)_{n=1}^{\infty})<\infty$ since for Bessel sequences we only know that $e((g_n)_{n=1}^{\infty})\geq \mbox{dim}\,\mbox{Ker}\,V^*$ (see \cite{CHet_al}).

This was also the reason for stating the theorem for frames, not for Bessel sequences.

After all, we have the following corollary.

\begin{corollary}\label{sesti glavni}
Let $(f_n)_{n=1}^{\infty}$ and $(g_n)_{n=1}^{\infty}$ be frames for $H$ with the analysis operators $U$ and $V$. If the operator $I-VU^*$ is compact, both $(f_n)_{n=1}^{\infty}$ and $(g_n)_{n=1}^{\infty}$ are near-Riesz bases.
\end{corollary}

\vspace{.1in}

Consider again a frame $(f_n)_{n=1}^{\infty}$ for $H$ with the analysis operator $U$. Recall that $(U^*)^{\dag}=U(U^*U)^{-1}$, since $U^*$ is a surjection. Hence, $U(U^*U)^{-1}U^*$ is the orthogonal projection to $\mbox{Im}\,U$ and, consequently, $I-U(U^*U)^{-1}U^*$ is the orthogonal projection to $(\mbox{Im}\,U)^{\perp}=\mbox{Ker}\,U^*$.

Suppose now that $e((f_n)_{n=1}^{\infty})<\infty$. Then we have $\mbox{dim}\,(\mbox{Ker}\,U^*)<\infty$; in other words, $I-U(U^*U)^{-1}U^*$ is a finite rank operator. Finally, observe that $U(U^*U)^{-1}$ is in fact the analysis operator of the canonical dual frame $(\tilde{f_n}=(U^*U)^{-1}f_n)_{n=1}^{\infty}$. This, together with the above corollary, shows that we have the following characterization of near-Riesz bases among all frames.

\begin{corollary} \label{sedmi glavni}
Let $(f_n)_{n=1}^{\infty}$ be a frame for $H$ with the analysis operator $U$. Then $(f_n)_{n=1}^{\infty}$ is a near-Riesz basis if and only if there exists a frame $(g_n)_{n=1}^{\infty}$ for $H$ with the analysis operator $V$ such that $I-VU^*$ is a compact operator.
\end{corollary}

Notice a formal similarity of Corollary \ref{sedmi glavni} and Corollary \ref{cetiri} (see also Definition \ref{ess dual}). The characterizing condition is formulated in the same way and the only difference is the order of the action of the analysis and the synthesis operators involved.

\vspace{.1in}

Recall from Proposition 5.5. in \cite{CHet_al} that for each frame $(f_n)_{n=1}^{\infty}$ we have $e((f_n)_{n=1}^{\infty})=\sum_{n=1}^{\infty}(1-\|\overline{f_n}\|^2)$, where
 $(\overline{f_n})_{n=1}^{\infty}$ is the Parseval frame canonically associated with $(f_n)_{n=1}^{\infty}$.
 In particular, if $e((f_n)_{n=1}^{\infty})<\infty$, this implies $\sum_{n=1}^{\infty}(1-\|\overline{f_n}\|^2)<\infty$. In the special case when
 $(f_n)_{n=1}^{\infty}$ is a Parseval frame with finite excess, we have $\sum_{n=1}^{\infty}(1-\|{f_n}\|^2)<\infty$.

 In this light it is natural to ask: can we describe those frames $(f_n)_{n=1}^{\infty}$ with the upper frame bound $B$ for which the series
 $\sum_{n=1}^{\infty}(B-\|{f_n}\|^2)$ converges?

 \begin{proposition}\label{convergence}
 Let $(f_n)_{n=1}^{\infty}$ be a frame for $H$ with the optimal upper frame bound $B$ such that the series
 $\sum_{n=1}^{\infty}(B-\|{f_n}\|^2)$ converges. Then $(f_n)_{n=1}^{\infty}$ is a near-Riesz basis.
 \end{proposition}
\begin{proof}
 We have to show that $e((f_n)_{n=1}^{\infty})<\infty$. Let $g_n=\frac{1}{\sqrt{B}}f_n,\,n\in \Bbb N$. Then $(g_n)_{n=1}^{\infty}$ is a frame
 for $H$ with the same excess. Its analysis operator is given by $U_1=\frac{1}{\sqrt{B}}U$ where $U$ is the analysis operator of
 $(f_n)_{n=1}^{\infty}$.

 Consider the canonical orthonormal basis $(e_n)$ for $l^2$. Then for all $n$ we have $$\la (I-U_1U_1^*)e_n,e_n\ra=1-\|U_1^*e_n\|^2=1-\|g_n\|^2=1-\frac{1}{B}\|f_n\|^2$$
 $$=
 \frac{1}{B}(B-\|f_n\|^2).$$ From this we get $$\mbox{tr}(I-U_1U_1^*)=\sum_{n=1}^{\infty}\la (I-U_1U_1^*)e_n,e_n\ra=\sum_{n=1}^{\infty}\frac{1}{B}(B-\|f_n\|^2)<\infty.$$ Thus, $I-U_1U_1^*$ is a trace class operator and hence compact. By Theorem \ref{peti glavni}, $(g_n)_{n=1}^{\infty}$ is a near-Riesz basis.
 \end{proof}

 \begin{remark}
 It is easy to see that for each frame $(f_n)_{n=1}^{\infty}$ we have $1-\|\overline{f_n}\|^2 \leq 1-\frac{1}{B}\|f_n\|^2,\,\forall n\in \Bbb N$.
 This, together with Proposition 5.5 from \cite{CHet_al}, serves as an alternative proof of Proposition \ref{convergence}. We omit the details.
 \end{remark}

 \vspace{.1in}

 In general, the converse of Proposition \ref{convergence} is not true. As an example, consider the canonical orthonormal basis $(e_n)$ for $l^2$ and the sequence $e_1,e_1,e_2,e_3,e_4,\ldots$. This is a frame with finite excess (equal to $1$) and the optimal upper frame bound $B=2$, but the series $\sum_{n=1}^{\infty}(2-\|f_n\|^2)$ diverges.

 Therefore, in order to ensure the convergence, one should add an extra condition.

 \begin{proposition}
 Let $(f_n)_{n=1}^{\infty}$ be a frame for $H$ with the optimal upper frame bound $B$ and the analysis operator $U$ such that $e((f_n)_{n=1}^{\infty})<\infty$ and $\mbox{dim}\,(\mbox{Im}( B \cdot I - U^* U ))<\infty$. Then $\sum_{n=1}^{\infty}(B-\|{f_n}\|^2)<\infty$.
 \end{proposition}
\begin{proof}
Let $g_n=\frac{1}{\sqrt{B}}f_n,\,n\in \Bbb N$. Then $(g_n)_{n=1}^{\infty}$ is a frame
 for $H$ with the same excess and the optimal upper frame bound equal to $1$. Its analysis operator is given by $U_1=\frac{1}{\sqrt{B}}U$, so
 $I-U_1^*U_1=\frac{1}{B}(B\cdot I-U^*U)$. Thus, by assumption, we have $\mbox{dim}\,(\mbox{Im}(I-U_1^*U_1))<\infty$. By Theorem \ref{drugi glavni},
 there exists a finite sequence $(x_n)_{n=1}^{k}$ such that $(x_n)_{n=1}^{k}\cup (g_n)_{n=1}^{\infty}$ is a Parseval frame for $H$.
 Since by assumption $(f_n)_{n=1}^{\infty}$ has finite excess, the excess of this new frame is also finite. Since this new frame is Parseval, Proposition 5.5 from \cite{CHet_al} implies $\sum_{n=1}^k(1-\|{x_n}\|^2)+\sum_{n=1}^{\infty}(1-\|{g_n}\|^2)<\infty$. From this we conclude
 $\sum_{n=1}^{\infty}(1-\|{g_n}\|^2)<\infty$. In terms of the original frame this means $\sum_{n=1}^{\infty}(1-\frac{1}{B}\|{f_n}\|^2)<\infty$ which after multiplication by $B$ gives us $\sum_{n=1}^{\infty}(B-\|{f_n}\|^2)<\infty$. 
 \end{proof}

\vspace{.1in}

 We end the paper with a comment concerning extensions of frames to Parseval frames. Recall from Remark \ref{nejedinstvenost} that a finite extension of a Bessel sequence to a Parseval frame (if it exists) is not unique. However, if we have a frame with finite excess that admits a finite extension to a Parseval frame, then the following corollary tells us that all such extensions have the same energy ({\em i.e.} $l^2$-norm).

\begin{corollary}\label{zadnji}
 Let $(f_n)_{n=1}^{\infty}$ be a frame for $H$ with the optimal upper frame bound equal to $1$ and the analysis operator $U$. Suppose that $e((f_n)_{n=1}^{\infty})<\infty$ and
 $\mbox{dim}\,(\mbox{Im}( I-U^*U))<\infty$. Let $(x_n)_{n=1}^{k}$ be any finite sequence in $H$ such that $(x_n)_{n=1}^{k} \cup (f_n)_{n=1}^{\infty}$ is a Parseval frame. Then
 $$
 \sum_{n=1}^k\|x_n\|^2=\sum_{n=1}^{\infty}(1-\|{f_n}\|^2)-e((f_n)_{n=1}^{\infty}).
 $$
\end{corollary}
\begin{proof}
Let $e((f_n)_{n=1}^{\infty})=e<\infty$. Since $(f_n)_{n = 1}^{\infty}$ is a near--Riesz basis, there exists a finite set $J \subseteq \Bbb N$ such that $(f_n)_{n \in \Bbb N \setminus J}$ is a Riesz basis. Theorem~$3.1.$ in~\cite{Holub} tells us that the number of vectors we have to remove from a near--Riesz basis in order to get a Riesz basis is equal to the excess of that near--Riesz basis. Therefore $|J| = e$ and using the same argument for the near--Riesz basis $(x_n)_{n = 1}^{k} \cup (f_n)_{n = 1}^{\infty}$ we get
$$
e \left( (x_n)_{n = 1}^{k} \cup (f_n)_{n = 1}^{\infty} \right) = e \left( (x_n)_{n = 1}^{k} \cup (f_n)_{n \in J} \cup (f_n)_{n \in \Bbb N \setminus J} \right) = k + |J| = k + e.
$$


On the other hand, since $(x_n)_{n=1}^{k} \cup (f_n)_{n=1}^{\infty}$ is a Parseval frame, Proposition 5.5 from \cite{CHet_al} implies
$e((x_n)_{n=1}^{k} \cup (f_n)_{n=1}^{\infty})=\sum_{n=1}^{k}(1-\|{x_n}\|^2)+\sum_{n=1}^{\infty}(1-\|{f_n}\|^2)$.

By comparing these two equalities, we obtain $k+e=\sum_{n=1}^{k}(1-\|{x_n}\|^2)+\sum_{n=1}^{\infty}(1-\|{f_n}\|^2)$ which gives us
$\sum_{n=1}^k\|x_n\|^2=\sum_{n=1}^{\infty}(1-\|{f_n}\|^2)-e$.
\end{proof}

\vspace{.2in}

\bibliographystyle{amsplain}

\begin{thebibliography}{99}

\bibitem{CHet_al} R.~Balan, P.G.~Casazza, C.~Heil, Z.~Landau, \textit{Deficits and excesses of frames}, Advances in Comp. Math.,
Special Issue on Frames, \textbf{18} (2003), 93-116.

\bibitem{Cas} P.G.~Casazza, \textit{The art of frame theory}, Taiwanese Journal of Math., Vol \textbf{4} (2) (2000) 129--202.

\bibitem{CL} P.G.~Casazza, N.~Leonhard, \textit{Classes of finite equal norm Parseval frames}, Contemp.~Math., \textbf{451} (2008) 11--31.

\bibitem{Chr} O.~Christensen, \textit{An introduction to frames and Riesz bases}, Birk\"{a}user, 2003.

\bibitem{CKK} O.~Christensen, H.O.~Kim, R.Y.~Kim, \textit{Extensions of Bessel sequences to dual pairs of frames}, Appl.~Comput.~Harmon.~Anal.,
\textbf{34} (2013), 224--233.

\bibitem{Dau} I.~Daubechies, \textit{Ten lectures on wavelets}, SIAM, Philadelphia 1992.

\bibitem{DS} R.J.~Duffin, A.C.~Schaeffer, \textit{A class of nonharmonic Fourier series}, Trans. Amer. Math. Soc., \textbf{72} (1952), 341--366.

\bibitem{H} P.R.~Halmos, \textit{A Hilbert space problem book}, Van Nostrand, 1982.

\bibitem{HW} C.E.~Heil, D.F.~Walnut, \textit{Continuous and discrete wavelet transforms}, SIAM Review, \textbf{31} (1989), 628--666.

\bibitem{Holub} J. Holub, \textit{Pre-frame operators, Besselian frames and near-Riesz bases in Hilbert spaces}, Proc. Amer. Math. Soc.,
\textbf{122} (1994), 779--785.

\bibitem{deng} D.F.~Li, W.C.~Sun, \textit{Expansion of frames to tight frames}, Acta Math. Sin. (Engl. Ser.), \textbf{25} (2009), 287--292.


%
%
%
%
%


\end{thebibliography}

\end{document}